\documentclass{elsart1}
\usepackage{amsmath}
\usepackage{amssymb}
\input{psfig.sty}

\def\stretch{\vartriangleleft\!\!\rightsquigarrow}
\newenvironment{proof}{\noindent {\it Proof.}\hspace{1pt}}
{\hfill$\qed$\bigskip}

\newtheorem{theorem}{Theorem}[section]
\newtheorem{lemma}{Lemma}[section]

\theoremstyle{definition}

\theoremstyle{remark}
\newtheorem{remark}{Remark}[section]

\begin{document}

\begin{frontmatter}

\title{Complex dynamics in a nerve fiber model with periodic
coefficients  }

\author[label1]{Chiara Zanini, Fabio Zanolin}
\address[label1]{
Dipartimento di Matematica e Informatica, Universit\`a,\\ via delle
Scienze 206, 33100 Udine, Italy
\\
mailto: chiara.zanini@dimi.uniud.it, \quad fabio.zanolin@dimi.uniud.it}

\begin{abstract}
\noindent We deal with the periodic boundary value problem for a
second-order nonlinear ODE which includes the case of the Nagumo
type equation $v_{xx} - g v + n(x) F(v) = 0,$ previously
considered by Grindrod and Sleeman and by Chen and Bell in the
study of nerve fiber models. In some recent works
we discussed the case of nonexistence of nontrivial solutions as
well as the case in which many positive periodic solutions
may arise, the different situations depending by threshold
parameters related to the weight function $n(x).$ Here we show
that for a step function $n(x)$ (or for small perturbations of it) it is
possible to obtain infinitely many periodic solutions and chaotic
dynamics, due to the presence of a topological horseshoe
(according to Kennedy and Yorke).
\end{abstract}


\date{}
\end{frontmatter}

\section{Introduction and statement of the main results}\label{sec-1}
This paper deals with the study of a class of nonlinear second order ODEs
which was proposed by Chen and Bell \cite{ChBe-94} as a model for a nerve fiber with
spines. In their model, the authors consider a nerve fiber as an infinitely long
cable. The spines, which are small protrusions on the membrane, are assumed to
be periodically distributed. Their population is characterized by a density
$n=n(x)$ and by a lumped ohmic resistance $1/\rho$ of the spine stem.
The current flowing through the spine stem is given by
$\rho(u-v),$ where $v(x,t)$ and $u(x,t)$ represent two membrane potentials
($v$ is the one corresponding to the dendrite,
while $u$ is that at the spine head). Chen and Bell, following the
cable theory approach previously considered by Baer and Rinzel
\cite{BaRi-91} and simplifying the active membrane dynamics from the
Hodgkin-Huxley form to the FitzHugh-Nagumo's, derived the following system
$$
\left\{
\begin{array}{ll}
v_t = v_{xx} - g v + n(x) \rho (u -v)\\
u_t = \rho (u -v) + f(u),\\
\end{array}
\right.
$$
where $g> 0$ is a constant such that $1/g$ represents the dendritic membrane resistance
and $f$ has the typical cubic form of the Nagumo equation \cite{Mk-70}. The search of the
steady state solutions $u=u(x), v= v(x)$ of the above system,
lead to the study of
$$
\left\{
\begin{array}{ll}
v_{xx} - g v + n(x) f(u) =0\\
v = S(u;\rho):= u - {\rho}^{-1} f(u).\\
\end{array}
\right.
$$
In \cite{ChBe-94}, the authors discuss, separately, the cases in which
the function $S$ is invertible or not (for a given $\rho$).
If $S$ is invertible and therefore we can write $u = {\tilde R}(v),$
for a suitable monotone increasing function ${\tilde R},$ we can
confine ourselves to the study
of the nonlinear second order scalar ODE
\begin{equation}\label{eq-1.1}
v_{xx} - g v + n(x) F(v) = 0,
\end{equation}
with $F(v):= f({\tilde R}(v)).$ We note that the possibility of inverting $S$
and therefore considering equation \eqref{eq-1.1} is not too restrictive
and can be assumed when $\rho$ is sufficiently large.
In \cite{ChBe-94} the function $n(x)$ is supposed to be piecewise constant
and $\beta$-periodic. More precisely, two positive constants $n_0$ (small)
and $n_1$ (large) are taken such that
$$
n(x)=\; \left\{
\begin{array}{ll}
n_1\,,\quad \mbox{for } \, 0\leq x < \alpha,\\
n_0\,,\quad \mbox{for } \, \alpha\leq x < \beta.
\end{array}
\right.
$$
In \cite[Lemma 4.1]{ChBe-94},
assuming further a condition which implies that $\int_0^1 F(s)\,ds > 0,$
it is stated that equation \eqref{eq-1.1} has a positive
periodic solution when $\alpha$ is sufficiently large (i.e., $\alpha$ is close to $\beta$).

In some recent papers \cite{ZaZa-05,ZaZa->} we extended the above recalled Chen and Bell
result toward different directions. In particular, in \cite{ZaZa-05}, we proved
(via a variational approach) that the existence of at least two positive periodic solutions
is guaranteed for an arbitrary weight function $n(x),$ provided that $\int_0^{\beta} n(x) \,dx$
is sufficiently large. On the other hand, in \cite{ZaZa->},
we obtained (via the Poincar\'e-Birkhoff fixed point theorem)
the existence of many positive periodic solutions under more restrictive
hypotheses on $n(x)$ (but general enough to include the piecewise constant case) and
removing the requirement that $\int_0^1 F(s)\,ds > 0.$
In this work we come back to the original assumptions on $n(x)$ of Chen and Bell's paper with
the aim to prove that their model presents solutions with a very complicated behavior.

In order to introduce our results, we consider equation \eqref{eq-1.1},
where, throughout the paper,
$g > 0,$ $n: {\mathbb R}\to {\mathbb R}^+:=[0,+\infty)$
is a $\beta$-periodic locally integrable function and
$F: {\mathbb R}\to {\mathbb R}$ is a $C^1$ function with a
$N$-shaped graph. More precisely, we assume that
there exists $a\in \,]0,1[$ such that
$$F(0) = F(a) = F(1) = 0$$
and, moreover,
$$
F(s) < 0 \; \mbox{ for } \,\,  0 < s < a \;
\mbox{ and  } \,\, s > 1, \quad  F(s) > 0\; \mbox{ for } \,\,  s
< 0 \; \mbox{ and  } \,\, a < s < 1.$$
A typical function $F$ satisfying such conditions
is given by the cubic nonlinearity
\begin{equation*}
F(s) = s(s-a)(1-s)
\end{equation*}
of the celebrated FitzHugh-Nagumo equation \cite{Ha-75,Ha-81}.
We look for solutions $v(x)$ of \eqref{eq-1.1} with
$$0 < v(x) < 1,\quad\forall\, x\in {\mathbb R}$$
and, in particular, we shall focus our attention to the
search of periodic solutions as well as to the detection
of solutions presenting some kind of chaotic-like behavior.

Starting with the 70's a great deal of theoretical and numerical results
have been concerned the investigation of nonlinear ordinary and partial differential
equations and systems associated to the FitzHugh-Nagumo's.
Variational or topological functional analytic
methods have been successfully applied in the study of various boundary
value problems associated to such equations. With this respect, we also
recall some recent works by Sweers-Troy \cite{SwTr-03} and Dancer-Yan
\cite{DaYa-06a,DaYa-06b} where an abundance of solutions with a peculiar
behavior for the Dirichlet or Neumann BVPs
is established for related (although different from ours) Nagumo type systems.
We stress, however, that some special features of equation \eqref{eq-1.1},
like the presence of a periodic weight, makes natural to address our
investigation toward different aspects that, perhaps, have not been
yet fully discussed in the literature.

In the present paper we propose a dynamical system approach.
To this end, we think at the one-dimensional $x$-variable
(representing the longitudinal axial dimension of the idealized
nerve fiber) as a temporal variable, in order to deal with
a first order system in the plane of the form
\begin{equation}\label{eq-1.3}
\left\{
\begin{array}{ll}
v' = y\\
y'= g v - n(t) F(v).\\
\end{array}
\right.
\end{equation}
We denote by $\psi$ the Poincar\'{e}'s operator related to system
\eqref{eq-1.3}, that is the map which associates to a point $z\in {\mathbb R}^2$
the value $\psi(z):=\zeta(\beta,z)$ of the solution $\zeta(\cdot)=\zeta(\cdot,z)$
of system \eqref{eq-1.3} with $\zeta(0) = z.$ We notice that,
without further assumptions on $F,$
we don't have $\psi$ necessarily defined on the whole plane ${\mathbb R}^2.$
On the other hand, by the smoothness of $F,$ we know that $\psi$
is a (orientation-preserving) homeomorphism of its domain
$D_{\psi}\subseteq {\mathbb R}^2$ onto its image $R_{\psi}\subseteq {\mathbb R}^2.$

Our main result is the following:

\begin{theorem}\label{th-1.1}
Assume that $F(\cdot)$ is strictly convex on $[0,a],$
strictly concave on $[a,1]$
and satisfies
\begin{equation}\label{eq-1.4}
\int_0^1 F(s)\,ds > 0.
\end{equation}
Define
\begin{equation}\label{eq-1.m}
m^*_0:= {g}\bigl/\bigl(\, {\max_{s\in [a,1]}
\tfrac{F(s)}{s}}\,\bigr)\bigr. .
\end{equation}
Let $0 < \alpha < \beta$ and
suppose that $n(\cdot)$ is the stepwise function given by
\begin{equation}\label{eq-1.5}
n(t):=\; \left\{
\begin{array}{ll}
n_1\,,\quad \mbox{for } \, 0\leq t < \alpha,\\
n_0\,,\quad \mbox{for } \, \alpha\leq t < \beta,
\end{array}
\right.
\end{equation}
with $0 < n_0 < n_1\,.$
Then there exists
${\omega}^* > m^*_0$
such that for every $n_0$ and $n_1$ with
$$0 < n_0 < m^*_0 < {\omega}^* < n_1\,,$$
there exist a compact invariant set
$$\Lambda\subseteq \,]a,1[\times\,]0,+\infty)$$
for which $\psi|_{\Lambda}$ is semiconjugate via a continuous surjection $h$
to a two-sided Bernoulli shift on two symbols.
Moreover, for each periodic sequence of two symbols $(i_k)_{k\in {\mathbb Z}}$
there is a point $z\in \Lambda$ which is periodic and such that $h(z) = (i_k)_{k\in {\mathbb Z}}$.
\end{theorem}

As we shall see at the end of the proof of Theorem \ref{th-1.1} (which
is performed in Section \ref{sec-4}) the semiconjugation of
$\psi|_{\Lambda}$ allows the following interpretation in terms
of the solutions to \eqref{eq-1.1}:

{\em
Consider a two-sided sequence of two symbols $\xi:=(i_k)_{k\in {\mathbb Z}}$
with
$$i_k \in \{1,2\},\quad\forall\, k\in {\mathbb Z}.$$
Then, there is at least one solution $v_{\xi}=v_{\xi}(x)$ of \eqref{eq-1.1}
which is defined for every $x\in {\mathbb R}$ and satisfies
$$0 < v_{\xi}(x) < 1,\quad\forall\, x\in {\mathbb R}$$
as well as
$(v_{\xi}(0),\tfrac {d}{dx} v_{\xi}(0))\in \Lambda.$
Moreover, the symbol $i_k =1$ means that
$v_{\xi}(x)$ has precisely two strict maximum points separated by
one strict minimum point along the time interval
$[(k-1)\beta,(k-1)\beta +\alpha],$
while, the symbol $i_k =2$ means that
$v_{\xi}(x)$ has precisely three strict maximum points separated by
two strict minimum points along the time interval
$[(k-1)\beta,(k-1)\beta +\alpha].$
In both the situations, $v_{\xi}(x)$ is convex in the interval $[(k-1)\beta +\alpha,k\beta],$
with $v'_{\xi}((k-1)\beta + \alpha) < 0$ and $v'_{\xi}(k\beta) > 0.$
If the sequence $(i_k)_{k\in {\mathbb Z}}$ is
periodic, that is $i_k = i_{k+\ell}$ for some $\ell \geq 1,$ then we can take
$v_{\xi}(\cdot)$ as a $\ell \beta$-periodic solution as well.
}

\medskip

A minor modification in the argument of the proof of Theorem \ref{th-1.1} yields to the following result.

\begin{theorem}\label{th-1.2}
Assume that $F(\cdot)$ is strictly convex on $[0,a],$
strictly concave on $[a,1]$
and satisfies \eqref{eq-1.4}.
Let
$m^*_0$
be like in \eqref{eq-1.m} and set
$$m^*_1:= \frac{g}{2\int_0^1 F(s)\,ds}\,.$$
Suppose that $n(\cdot)$ is
defined as in \eqref{eq-1.5} with
$$0 < n_0 < m^*_0 < m^*_1 < n_1\,.$$
Then there exist $\alpha^* >0$ and $\delta > 0$ such that
for every $\alpha$ and $\beta$ with
$$\alpha^* < \alpha < \beta < \alpha +\delta,$$
the same conclusion of Theorem \ref{th-1.1} holds.
\end{theorem}

We observe that it is possible to obtain extensions of Theorem \ref{th-1.1}
and Theorem \ref{th-1.2} by producing symbolic dynamics on $p$ objects
with $p > 2$ (see Remark \ref{rem-4.1}).

\section{Notation and basic tools}\label{sec-2}
Throughout the paper we denote by ${\mathbb R}$, ${\mathbb R}^+$ and
${\mathbb R}^+_{0}$ the sets of real, real nonnegative and positive numbers.
By $||\cdot||$ me mean a given norm (for instance, the euclidean one) in ${\mathbb R}^2.$

In the sequel we consider some special subsets of the plane, called {\em generalized rectangles}.
By such a name we mean any subset of ${\mathbb R}^2$ which is homeomorphic to the unit square
$[0,1]^2.$ It is possible to define an orientation for a generalized rectangle ${\mathcal R}$,
by selecting two disjoint compact sub-arcs ${\mathcal R}_{l}^-$ and ${\mathcal R}_{r}^-$ of its boundary,
called the left and the right sides of ${\mathcal R}$. In a more formal way,
using the Jordan-Shoenflies theorem, we can choose a homeomorphism $h$ of the plane onto itself
such that
$$h([0,1]^2) = {\mathcal R},\quad h(\{0\}\times [0,1]) = {\mathcal R}_{l}^-\;\; \mbox{and }\;
h(\{1\}\times [0,1]) = {\mathcal R}_{r}^-\,.$$
In this situation, we also set
$${\mathcal R}^- := {\mathcal R}_{l}^- \cup {\mathcal R}_{r}^-\,.$$
By a path $\gamma$ we mean a continuous map $\gamma: [0,1]\to {\mathbb R}^2$.

Let $({\mathcal A},{\mathcal A}^-)$ and $({\mathcal B},{\mathcal B}^-)$ be oriented rectangles
and suppose that
$\psi: {\mathbb R}^2 \supseteq D_{\psi} \to {\mathbb R}^2$ is a map which is continuous on a set
${\mathcal D}\subseteq D_{\psi}\cap {\mathcal A}$.
We say that $({\mathcal D},\psi)$ {\em stretches $({\mathcal A},{\mathcal A}^-)$ to $({\mathcal B},{\mathcal B}^-)$
along the paths} and write
$$({\mathcal D}, \psi) : ({\mathcal A}, {\mathcal A}^-) \stretch
({\mathcal B}, {\mathcal B}^-),$$
if, for every path $\gamma : [0,1]\to {\mathcal A}$ with $\gamma(0) \in {\mathcal A}_{l}^-$
and $\gamma(1) \in {\mathcal A}_{r}^-$, there exists a subinterval $[t_1,t_2]\subseteq [0,1]$
such that
$$\gamma(t) \in {\mathcal D},\quad \psi(\gamma(t)) \in {\mathcal B},\;\;\forall\,  t\in [t_1,t_2]$$
and
$\psi(\gamma(t_1)),$ $\psi(\gamma(t_2))$ belong to different components of ${\mathcal B}^-.$

It is easy to check that this property of stretching along the paths is preserved under the
compositions of maps. Moreover, as proved in \cite{PaZa-04a,PaZa-04b}, if
$$({\mathcal D}, \psi) : ({\mathcal R}, {\mathcal R}^-) \stretch
({\mathcal R}, {\mathcal R}^-),$$
then there exists at least one fixed point for $\psi$ in ${\mathcal D}.$ Such a fixed point property,
when applied to different subsets of the domain and to the iterates of $\psi$, may be exploited
in order to find many different periodic points for the map $\psi.$

The key tool that we use for the proof of the existence of chaotic-like dynamics for the Poincar\'{e}'s map
associated to equation \eqref{eq-1.1} is the following lemma which combines the above mentioned
fixed point theorem with the theory of topological horseshoes developed by Kennedy and Yorke in
\cite{KeYo-01} (see also \cite{KeKoYo-01}).

\begin{lemma}\label{lem-2.1}
Let $({\mathcal R},{\mathcal R}^-)$ be an oriented rectangle
and let
$\psi: {\mathbb R}^2\supseteq D_{\psi}\to {\mathbb R}^2$ be a map.
Assume there exist compact sets
${\mathcal D}_{i}\subseteq {\mathcal R}\cap D_{\psi}$ (for $i = 1,2$)
with
$${\mathcal D}_{1}\cap {\mathcal D}_{2} = \emptyset,$$
such that
$$({\mathcal D}_{i},\psi): ({\mathcal R},{\mathcal R}^-)  \stretch ({\mathcal R},{\mathcal R}^-)\,,
\quad\forall\, i= 1,2.$$
Then the following conclusions hold:
\begin{itemize}
\item{} For any two$-$sided sequence of two symbols $(i_k)_{k \in {\mathbb Z}}\in
\Sigma_2:=\{1,2\}^{\mathbb Z},$
there exists a sequence $(w_k)_{k\in {\mathbb Z}}$
such that $w_k\in {\mathcal D}_{i_k}$
and $\psi(w_k) = w_{k+1}\,,$ for all ${k\in {\mathbb Z}}\,;$
\item{} If the sequence
$(i_k)_{k \in {\mathbb Z}}\in \Sigma_2$
is $m$-periodic ($m\geq 1$), then there exists a corresponding sequence
$(w_k)_{k\in {\mathbb Z}}$
with $w_k\in {\mathcal D}_{i_k}$ which is $m$-periodic.
\end{itemize}
Furthermore, as a consequence of the above properties,
there exists a nonempty compact set
$$\Lambda\subseteq {\mathcal D}_1 \cup {\mathcal D}_2$$
which is invariant for $\psi$
and such that $\psi|_{\Lambda}$ is semiconjugate to the two-sided Bernoulli shift on
two symbols.
The subset of $\Lambda$ made by the periodic points of $\psi$ is dense in
$\Lambda$ and the counterimage (by the semiconjugacy) of any periodic sequence in $\Sigma_2$
contains a periodic point of $\psi$.
\end{lemma}

For the proof, see \cite{PaZa-04a,PaZa-04b} and \cite{PiZa-07} for more recent details.
Note that the semiconjugation to the Bernoulli shift and the density of periodic points are typical
requirements for chaotic dynamics.
The kind of chaotic-like dynamics described by Lemma \ref{lem-2.1}
is based on the definition of chaos in the coin-tossing sense
by Kirchgraber and Stoffer \cite{KiSt-89}.
The same chaotic behavior
is also obtained by Kennedy, Ko{\c{c}}ak e Yorke in \cite[Proposition 5]{KeKoYo-01}.
With respect to \cite{KeKoYo-01} and \cite{KiSt-89},
our case takes into account also of the presence of periodic itineraries
generated by periodic points. Concerning the applications to second order ODEs with periodic
coefficients,
where $\psi$ is the Poincar\'{e} map associated to an
equivalent first order differential system in the plane,
the complex dynamics we obtain is in line with similar results
appeared in the literature for different type of equations
(see, for instance Capietto, Dambrosio and Papini \cite{CaDaPa-02}).
Finally, we refer to Mischaikow and Mrozek \cite{MiMr-95} and to Zgliczy\'nski and Gidea \cite{ZgGi-04}
for related topological approaches in higher dimension.

\section{Technical lemmas}\label{sec-3}

Let us consider the equation
\begin{equation}\label{eq-3.1}
x'' - g x + n(t)F(x) = 0 \qquad (x':= \frac{dx}{dt} ),
\end{equation}
as well as the associated first order system in the phase-plane
\begin{equation}\label{eq-3.2}
\left\{
\begin{array}{ll}
x' = y\\
y' = g x - n(t) F(x),
\end{array}
\right.
\end{equation}
where $n(t)\geq 0$ is a $\beta$-periodic piecewise continuous function
that will be defined in a more precise manner in the sequel.

The map $F: {\mathbb R}\to {\mathbb R}$ is a $C^1$ function with a $N$-shaped
graph. In particular, we assume there exists $a \in \, ]0,1[\,$
such that
$$F(0) = F(a) = F(1) = 0$$
and, moreover,
$$
F(s) < 0 \; \mbox{ for } \,\,  0 < s < a \;
\mbox{ and  } \,\, s > 1, \quad  F(s) > 0\; \mbox{ for } \,\,  s
< 0 \; \mbox{ and  } \,\, a < s < 1.$$
A typical example for $F$ is given by the Nagumo type cubic nonlinearity
\begin{equation}\label{eq-3.3}
F(s) = s(s-a)(1-s).
\end{equation}
For the moment and in order to simplify the subsequent discussion,
we assume \eqref{eq-3.3}. We point out, however, that the properties we are
going to present below are still true for a broader class of functions.

We are looking for solutions $x(\cdot)$ to \eqref{eq-3.1} such that
\begin{equation}\label{eq-bound1}
0 \leq x(t) \leq 1,\quad \forall\, t\in {\mathbb R}.
\end{equation}
Accordingly, these solutions are the same also for any other equation where the
function $F(s)$ is the same like the given one in $[0,1]$ but possibly different elsewhere.
For convenience we then suppose that
$$|F(s)|\leq 1,\quad\, \forall\, s\in {\mathbb R}\setminus \,[0,1]$$
so that all the solutions of equation \eqref{eq-3.1} are globally defined.
Clearly, at the conclusion of our argument, we need to check that the solutions
we are interested in satisfy the constraint in \eqref{eq-bound1}.
To this purpose, we state the following lemma whose proof (which is valid also
for a more general class of equations) is postponed in the Appendix.
\begin{lemma}\label{lem-3.0}
Let $n: [t_0,t_1]\to {\mathbb R}^+$ be a Lebesgue integrable function
and suppose that
$F_0: {\mathbb R}\to {\mathbb R}$ is a locally Lipschitz function
such that
$$F_0(s) > 0 \, \mbox{ for } \, s < 0,\quad F_0(s) < 0  \, \mbox{ for } \, s> 1
\, \mbox{ and } \, F_0(s) = F(s),\;\forall\,
s\in [0,1].$$
Let $x(\cdot)$ be a solution of
$$x'' - g x + n(t)F_0(x) = 0,\quad (g\geq 0)$$
defined on $[t_0,t_1]$ and such that
$$0 < x(t_0), x(t_1) < 1.$$
Then,
$$0 < x(t) < 1,\quad\forall\, t\in [t_0,t_1]$$
and, therefore, $x(\cdot)$ is a solution of \eqref{eq-3.1}.
\end{lemma}

\bigskip

As a preliminary discussion, we start by a phase-plane analysis of
the trajectories of \eqref{eq-3.2} in the case when $n(\cdot) =
\mbox{constant} = \mu > 0.$ Accordingly, we study the autonomous system
$$
\left\{
\begin{array}{ll}
x' = y\\
y' = g x - \mu\, F(x),
\end{array}
\right.  \leqno{(E)}
$$
and look at the effects of the parameter $\mu$ on the qualitative
behavior of the orbits.

\bigskip

First of all we observe that there is $m^*_0$ such that,
for $\mu\in\, ]0,m^*_0[\,$
$$g s - \mu F(s) > 0,\quad \forall \, s > 0,$$
while for each
$\mu > m^*_0$ there are two nontrivial equilibrium points
$P=(a_{\mu},0)$ and $Q=(c_{\mu},0)$ to system $(E)$ with
$$0 < a < a_{\mu} < c_{\mu} < 1$$
(see, for instance, Fig. 1 for an example which depicts this situation).
Moreover, $a_{\mu}\to a^+$ and $c_{\mu}\to 1^-$ as $\mu\to
+\infty.$ The point $P$ is a center and $Q$ is a saddle.
This is true for $F(s)$ defined in \eqref{eq-3.3} as well as for any function $F$
which is strictly concave in $[a,1].$

The actual value of $m^*_0$ can be computed as
$$m^*_0 = \frac{g}{F'(s^*)}\,,$$
where $s^* \in\, ]a,1[$ solves the equation $sF'(s) = F(s)$
yielding the maximal slope.

\bigskip

\begin{figure}[h]
\quad\psfig{file=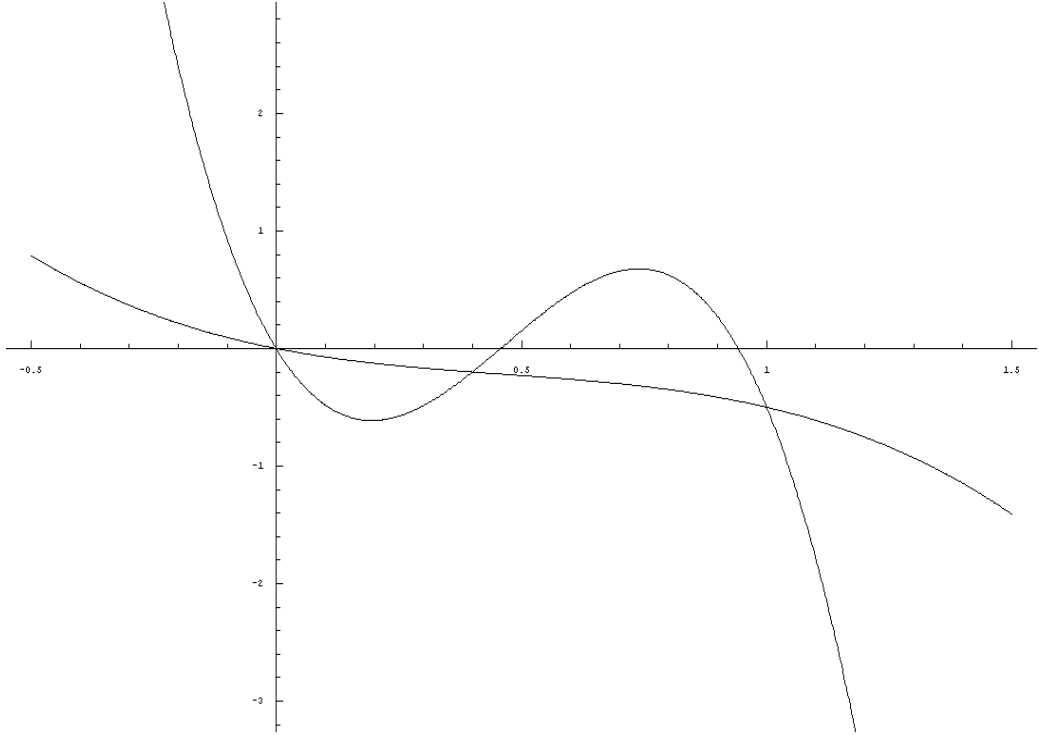,width=5.4in,height=4.2in,angle=0}
\vspace*{0in} \caption{\footnotesize{\em Different shapes for the graph of the function $y = - g s + \mu F(s),$
for a small and a large value of $\mu.$ In this example, we have chosen
$g=0.5$ and $F(x) = x(1-x)(x-a),$ with $a=0.4.$ The graph of the decreasing function
corresponds to the case $\mu = n_0 = 0.8,$ while the other graph (with two humps) is obtained
for $\mu = n_1 = 16.$}}
\vspace{0.4cm}
\end{figure}

\bigskip

System $(E)$ is of conservative type with energy
\begin{equation}\label{eq-3.en}
{\mathcal E}(x,y) = {\mathcal E}^{\mu}(x,y):= \frac{1}{2} y^2 - g
\frac{1}{2} x^2 + \mu {\mathcal F}(x),
\end{equation}
where
$$
{\mathcal F}(x):= \int_0^x F(s)\,ds.
$$
We denote by $L_c$ the part of the energy line at level $c$ lying
in the strip $0\leq x\leq 1,$ that is
$$
L_c:=\{(x,y): {\mathcal E}(x,y) = c,\; 0\leq x \leq 1\}.
$$

\bigskip

\begin{figure}[h]
\quad\psfig{file=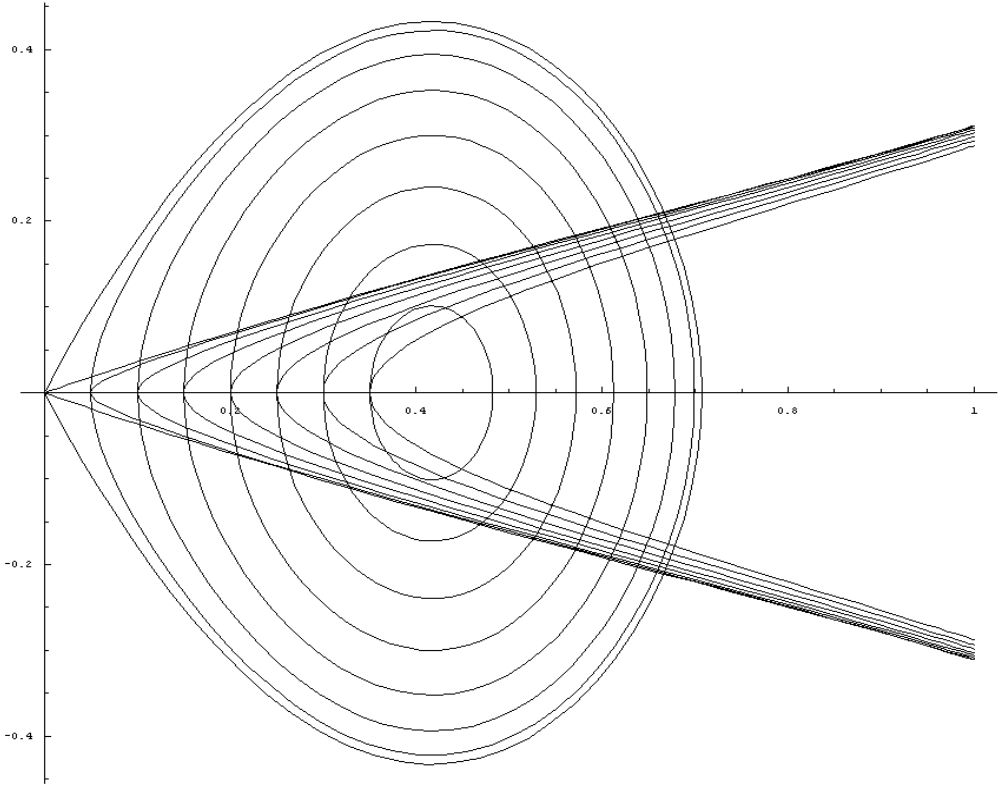,width=5.4in,height=4.2in,angle=0}
\vspace*{0in} \caption{\footnotesize{\em Level lines (at different energy) of
system (E) with $\mu =10$ (the closed lines) overlapped with those
of the same system for $\mu=0.1.$ For this example, we have chosen
$g=0.1$ and $F(x) = x(1-x)(x-a),$ with $a=0.4.$}}
\vspace{0.4cm}
\end{figure}

We present now some technical results which provide dynamical information about the
trajectories of system $(E)$ for different values of the parameter $\mu.$

\begin{lemma}\label{lem-3.1}
Suppose that $\mu \in \,]0,m^*_0[\,.$ Then, for every $w_0 = (p_0,0)$ with
$0\leq p_0 < 1$ the level line $L_c$ ($c:= {\mathcal E}(w_0)$)
passing through $w_0$ intersects any vertical line $x = \xi$ with $p_0 < \xi \leq 1$
exactly at two points
$(\xi,-y_{\xi}),$ $(\xi,y_{\xi}).$ For $0 < p_0 < 1,$ $L_c$ is an orbit path
and the time needed to move from $(\xi,-y_{\xi})$
to $w_0$ (which coincides to the time needed to move from $w_0$ to $(\xi,y_{\xi})$)
along the orbit is given by
$$
\sigma(p_0,\xi) = \int_{p_0}^{\xi} \frac{ds}
{\sqrt{g(s^2 - p_0^2) - 2\mu ({\mathcal F}(s) - {\mathcal F}(p_0))}}\,.
$$
Assume that
$$
\int_0^1 F(s)\, ds > 0.
\leqno{(H_0)}
$$
Then, there is $m^*_1 > m^*_0$ such that, for each $\mu >
m^*_1\,,$ the level line $L_0$ is a homoclinic orbit of system
$(E)$ intersecting the abscissa at a point $(b_{\mu},0)$
with
$$a_{\mu} < b_{\mu} < c_{\mu}\,.$$
Moreover, for each $c\in \,]{\mathcal E}(P),0[\,,$ the level line $L_c$ is a
closed curve which corresponds to a periodic orbit of system
\eqref{eq-3.2} having minimal period
$$
\tau(x_0) = \frac{2}{\sqrt{\mu}}\, \int_{x_0}^{x_1} \frac{ds}
{\sqrt{2({\mathcal F}(x_0) - {\mathcal F}(s)) -
\frac{g}{\mu}(x_0^2 - s^2)}}\,,
$$
where $x_0$ and $x_1,$ satisfying
$$0 < x_0 < a_{\mu} < x_1 < b_{\mu}\,,$$
are the two solutions of the equation ${\mathcal E}(x,0) = c$ with $x\in\, ]0,1[\,.$
\end{lemma}
\begin{proof}
We start by considering system $(E)$ for $\mu \in \, ]0,m^*_0[\,.$ Let
$L_c$ be a level line passing through a point $w_0 = (p_0,0)$ with $0\leq p_0 < 1.$
By definition, we have
$$y^2 = {\mathcal G}(x):=
g(x^2 - {p_0}^2) - 2\mu \bigl({\mathcal F}(x) -{\mathcal F}({p_0}) \bigr),\quad
x\in [0,1].$$
The function ${\mathcal G}$ is strictly increasing on $[0,1].$ In fact
${\mathcal G}'(x) = 2(g x - \mu F(x)) > 0,$ $\forall\, x\in\,]0,1].$
Since ${\mathcal G}(p_0)=0,$ it turns out that $\sqrt{{\mathcal G}(x)}$ is defined
on $[p_0,1]$ and therefore $L_c$ intersects any vertical line $x=\xi$
(for $p_0 < \xi\leq 1$) exactly at the two points $(\xi,-y_{\xi})$
and $(\xi,y_{\xi})$ with
$y_{\xi} = \sqrt{{\mathcal G}(\xi)}.$ From the relation
$$\frac{x'(t)}{\sqrt{{\mathcal G}(x(t))}} = 1,$$
which is satisfied for any time interval $I$ in which $y(t) = x'(t) >  0,$
we can compute the time $\sigma(p_0,\xi)$ at which a solution of system $(E)$
(starting at the point $w_0$ for $t=0$) reaches the point $(\xi,y_{\xi}).$
Actually, this is the standard time-mapping formula which reads in our case as
$$\sigma(p_0,\xi)=\int_{0}^{\sigma(p_0,\xi)}\frac{x'(t)}{\sqrt{{\mathcal G}(x(t))}}\,dt
= \int_{p_0}^{{\xi}}\frac{d x}{\sqrt{{\mathcal G}(x)}}\,.$$
This proves the first part of our lemma.
\\
We consider now system $(E)$ for $\mu > m^*_0\,.$ In this case, we have two equilibrium
points $P=(a_{\mu},0)$ and $Q=(c_{\mu},0),$ with
$a < a_{\mu} < c_{\mu} < 1.$
The level line $L_0$ passing through the origin is defined by the relation
$$y^2 = {\mathcal H}(x):=g x^2 - 2\mu{\mathcal F}(x),\quad x\in [0,1].$$
The function ${\mathcal H}$
is strictly increasing and positive on $\,]0,a_{\mu}],$ strictly decreasing on
$[a_{\mu},c_{\mu}]$ and strictly increasing on $[c_{\mu},1].$
Hence, ${\mathcal H}(1) = g - 2\mu {\mathcal F}(1) >
g c_{\mu}^2- 2\mu {\mathcal F}(c_{\mu}) = {\mathcal H}(c_{\mu}).$
Since ${\mathcal F}(1) = \int_0^1 F(s)\,ds > 0$
(by assumption $(H_0)$ ) we know that for $\mu$ sufficiently large,
say
$\mu > \frac{g}{2{\mathcal F}(1)}\,,$
the function ${\mathcal H}$ has exactly one zero in $\,]0,c_{\mu}[$
and such zero, denoted by $b_{\mu}\,,$ is contained in $\,]a_{\mu},c_{\mu}[\,.$
Thus, if we define $m^*_1$ as
$$m^*_1:= \frac{g}{2\int_0^1 F(s)\,ds}\,,$$
we have that for every $\mu > m^*_1,$ the level
line $L_0$ is a close curve (the locus of a homoclinic trajectory of $(E)$ )
intersecting the abscissa at the point $(b_{\mu},0),$ with $a_{\mu} < b_{\mu} < c_{\mu}\,.$
A similar argument shows that, for every $c \in \, ]{\mathcal E}(P),0[\,,$
the level line $L_c$ is a closed curve surrounding the equilibrium point $P$
and, as a matter of fact, it is the locus of a periodic orbit.
\\
In this case, the energy relation reads as
$$y^2 = 2 c + {\mathcal H}(x),\quad x\in [0,1].$$
The equation ${\mathcal H}(x) = - 2c,$ with $c\in \, ]{\mathcal E}(P),0[\,$
and $x\in \,]0,1[$ has exactly two solutions, $x_0$ and $x_1\,,$ with
$$0 < x_0 < a_{\mu} < x_1 < b_{\mu}\,.$$
We rewrite now the energy relation as
$$y^2 = - {\mathcal H}(x_0) + {\mathcal H}(x) = \mu \bigl( 2
({\mathcal F}(x_0) - {\mathcal F}(x)) - \frac g {\mu} (x_0^2 - x^2) \bigr ).$$
Arguing like in the first part of this proof (the case of $\mu$ small)
we can compute again the time-mapping related to the orbit $L_c\,.$
We find that the fundamental period $\tau(x_0)$ of such orbit is twice the time
to move along $L_c$ from $(x_0,0)$ to $(x_1,0)$
in the upper half plane $y = x'> 0.$ Thus, integrating
$$\frac{x'(t)}{\sqrt{- {\mathcal H}(x_0) + {\mathcal H}(x(t))}}\,=1,$$
we easily obtain the expression for $\tau(x_0).$ This concludes the proof.
\end{proof}

\begin{remark}\label{rem-3.1}
The choice of $m^*_1$ given in the proof of Lemma \ref{lem-3.1}
is not the optimal one. In fact, one could define
$m^*_1$ as the infimum of the $\mu > m^*_0$ such that the level line $L_0$
intersects the abscissa at some point in $\,]a_{\mu},c_{\mu}[\,.$
\end{remark}

\medskip
The next two lemmas are stated without the corresponding proofs which can be
inferred from Lemma \ref{lem-3.1} via elementary considerations.

\begin{lemma}\label{lem-3.2}
Let us define
$$\Lambda:= \sup_{0 < s\leq 1} \frac{F(s)}{s}\,,\quad
\Theta:= \sup_{0 < s\leq 1} \frac{-F(s)}{s}\,.$$ Suppose that $\mu
\in \,]0,m^*_0[\,.$ Then, for every $p_0$ with $0 < p_0 < 1$ and
$\xi$ with $p_0 < \xi \leq 1$ we have
$$
\frac{1}{\sqrt{g + \mu\,\Theta}}\,{\cosh}^{-1}(\tfrac{\xi}{p_0})\leq \sigma(p_0,\xi)
\leq \frac{1}{\sqrt{g - \mu\,\Lambda}}\,{\cosh}^{-1}(\tfrac{\xi}{p_0}).
$$
\end{lemma}
\noindent
Note that for $F$ like in \eqref{eq-3.3} it holds that $\Theta = |F'(0)|.$
Moreover, $g - \mu\Lambda > 0$ for $\mu < m^*_0\,.$

\begin{lemma}\label{lem-3.3}
Assume $(H_0)$ and let $\mu > m^*_1$ (with $m^*_1$ as in Lemma \ref{lem-3.1}).
Then, for every $x_0$ with
$$0 < x_0 < a,$$
it follows that
$$
\lim_{\mu\to +\infty}\tau(x_0) {\sqrt{\mu}} = \sqrt{2}\, \int_{x_0}^{x_0^+} \frac{ds}
{\sqrt{{\mathcal F}(x_0) - {\mathcal F}(s)}}\,,
$$
where $x_0^+ \in\, ]a,1[$ is such that
$$
\int_{x_0}^{x_0^+} F(s)\,ds = 0.
$$
\end{lemma}

\bigskip

At this point, we come back to the non-autonomous system
\eqref{eq-3.2} and define the function $n(t)$ as in
\cite{ChBe-94}. Accordingly, we set
\begin{equation}\label{eq-3nt}
n(t):=\; \left\{
\begin{array}{ll}
n_1\,,\quad \mbox{for } \, 0\leq t < \alpha,\\
n_0\,,\quad \mbox{for } \, \alpha\leq t < \beta,
\end{array}
\right.
\end{equation}
with
$$0 < n_0 < n_1\,.$$
In the sequel we treat $n_0$ as a small number and $n_1$ as a
large one. More precisely, we shall always assume
$0 < n_0 < m^*_{0} < m^*_{1} < n_1\,.
$
However, some further conditions on $n_0$ and $n_1$ will be required when needed.

\medskip
In order to study equations \eqref{eq-3.1} and \eqref{eq-3.2} we analyze the system
$$
\left\{
\begin{array}{ll}
x' = y\\
y' = g x - n_i\, F(x),
\end{array}
\right.
\qquad i= 0, 1
\leqno{(E_i)}
$$
and follow its trajectories for $i=1$ along a time interval of length $\alpha$
and for $i=0$ along a time interval of length $\beta - \alpha.$
Clearly, $(E_0)$ and $(E_1)$ are nothing but two aspects of the previously studied
system $(E)$ with $\mu = n_0$ or $\mu = n_1\,,$ respectively.
Also the energy function ${\mathcal E} = {\mathcal E}^{\mu}$ defined in
\eqref{eq-3.en} takes two different (but similar) forms that we denote by
${\mathcal E}_0$ and ${\mathcal E}_1\,,$ i.e.,
$$
{\mathcal E}_i(x,y) = {\mathcal E}^{n_i}(x,y)= \frac{1}{2} y^2 - g
\frac{1}{2} x^2 + n_i {\mathcal F}(x),\quad i = 0,1.
$$
The corresponding level lines $L^i_c$ (contained in the strip $[0,1]\times{\mathbb R}$)
are defined consequently.

Let us fix a value
$$c \in\,]{\mathcal E}_1(a_{n_1},0),0[$$
and consider the compact annular region
$${\mathcal M}^1_c:= \{(x,y): 0\leq x \leq 1, \;  c\leq {\mathcal E}_1(x,y) \leq 0\}$$
surrounding the point
$$P_1:= (a_{n_1},0).$$
The set ${\mathcal M}^1_c$ is invariant for the dynamical system generated by $(E_1)$
and, therefore, we can use the Pr\"{u}fer transformation and express the solutions
of $(E_1)$ with initial value in ${\mathcal M}^1_c$ using polar coordinates with center
at $P_1\,.$ Accordingly, we define by $\theta(t,z)$ and $\rho(t,z)$ the angular and the radial
coordinates of the solution  $(x(t),y(t))$
of $(E_1)$ with $(x(0),y(0)) = z\in {\mathcal M}^1_c\,.$

\bigskip
For each point $z$ in the set
$${\mathcal N}_c:= \{(x,y): a_{n_1}\leq x \leq 1, \; y \geq 0,\; c
\leq {\mathcal E}_1(x,y) \leq 0\}
={\mathcal M}^1_c\cap (\, [a_{n_1},+\infty)\times[0,+\infty)\,)$$
we can fix the angular coordinate in order to have
$$\theta(0,z)\in [0,\tfrac{\pi}{2}].$$

The next result rephrases the conclusions from Lemma \ref{lem-3.1}
and Lemma \ref{lem-3.3} in terms of the angular coordinate.
Indeed, we have:

\begin{lemma}\label{lem-3.5}
Assume $(H_0)$ and let $n_1  > m^*_1\,.$ Then
the following properties hold:
\begin{itemize}
\item{}
$\theta(\alpha,z) > - \pi, \quad \forall  z\in
L_0\cap {\mathcal N}_c\,;$
\item{} for each $\varepsilon > 0,$ for each $k\in
{\mathbb N}$ and for each $c\in \,[-\varepsilon,0[\,,$ there is
$n^*= n^*_{k,c}\,,$ such that for every $n_1 \geq n^*,$
$\theta(\alpha,z) < -\tfrac{\pi}{2} - 2k\pi,\quad
\forall  z\in L_c\cap {\mathcal N}_c\,.$
\end{itemize}
\end{lemma}
\begin{proof}
Assumption $(H_0)$ and the choice $n_1 > m^*_1$ ensure that the level line $L_0$ is homoclinic to the origin
for system $(E_1)$ (see Lemma \ref{lem-3.1}).
Thus, if we take $z\in L_0\cap {\mathcal N}_c\,,$ the solution $(x(t),y(t))$
departing from $z$ moves along $L_0$ in the clockwise sense and cannot reach the origin at any finite time.
Accordingly, as long as we run system $(E_1)$ we have $\theta(t,z) > - \pi$, for every $t\geq 0$.
This proves the first assertion of the lemma.
On the other hand, if $z\in L_c\cap {\mathcal N}_c\,,$ for $c < 0$ (in particular $z\not\in L_0$)
we know that $z$ belongs to a closed orbit of system $(E_1)$ whose fundamental period $\tau(z)$ tends to zero
as $\mu=n_1$ grows to infinity (see Lemma \ref{lem-3.3}). As a consequence, during a time interval
of length $\alpha > 0,$ the trajectory will make at least
$\displaystyle{\lfloor \tfrac{\alpha}{\tau(z)}\rfloor}$ turns around $P_1$. From this fact, the
second assertion of the lemma easily follows.
\end{proof}

We have now at hand almost all the needed tools to construct a domain
containing a topological horseshoe. First, however, we need to prove another
technical result.

\bigskip
Suppose that $n_0 > 0$ is a given constant satisfying
$$n_0 < m^*_0\,.$$
We consider system $(E_0)$ and look at the orbits of such a system passing through a
point $(p_0,0)$ with $0 < p_0 < 1.$ Recall also the definition of the function
$\sigma$ in Lemma \ref{lem-3.1} that we call now
$$
\sigma_0(p_0,\xi):= \int_{p_0}^{\xi} \frac{ds}
{\sqrt{g(s^2 - p_0^2) - 2n_0 ({\mathcal F}(s) - {\mathcal F}(p_0))}}\,,\quad \xi \in \, ]p_0,1].
$$
The level line $L^0_c$ ($c:= {\mathcal E}_0(p_0,0)$) crosses the vertical line $x=\xi$
exactly at two points and $\sigma_0(p_0,\xi)$ is half of the time needed to move
from one of these intersection points to the other along the orbit of $(E_0)$
and passing through $(p_0,0).$

\medskip

First of all, we observe that there is a $\check{p}_0\in \,]0,a[$ such that
\begin{equation}\label{eq-3pck}
\sigma_0(p_0,a)> \frac{\beta-\alpha}{2}\,,\quad
\mbox{for every } p_0\in \,]0,\check{p}_0].
\end{equation}

\medskip
Next, assuming $(H_0)$,
we consider system $(E_1)$ for a general $n_1 > m^*_1$ and
consider a level line $L^1_c$ ($c:= {\mathcal E}_1(p_0,0)$)
passing through a point $(p_0,0)$ with
$$0 < p_0 < a.$$
We denote by $(p_1,0)$ with
$$a_{n_1} < p_1 < b_{n_1}$$
the other
intersection of such a level line with the $x$-axis. Note that
$p_1$ depends by $n_1$ even if, for sake of simplicity in the
notation, we do not make this fact explicit.

We claim that there exist $\hat{p}_0\in \,]0,a[$ and $m^*_2 > m^*_1$
such that if we choose
$$p_0 \in  \, ]0,\hat{p}_0[$$
and
$$n_1 > m^*_2\,,$$
then
\begin{equation}\label{eq-3phat}
\sigma_0(p_1,b_{n_1})< \frac{\beta-\alpha}{2}\,.
\end{equation}
To prove this claim, we need to introduce some further notation.
By the sign conditions on $F(s)$ and $(H_0)$, we know that there exists
a (unique) value
$$b \in\, ]a,1[$$
such that
$$\int_0^b F(s)\, ds = 0.$$
With this position we have that
$$b_{\mu} > b, \quad\forall\, \mu > m^*_1$$
and $b_{\mu} \to b$ as $\mu\to+\infty$ with the function $\mu\to b_{\mu}$
decreasing on $\,]m^*_1,+\infty).$ Thus, if we fix a value
$$\bar{\mu} > m^*_1\,,$$
we can conclude that
$$b \leq b_{\mu} \leq b_{\bar{\mu}}\,,\quad\forall\, \mu \geq \bar{\mu}\,.$$
Hence, we can define
$$\eta:= \min_{s\in [b,b_{\bar{\mu}}]}F(s) > 0.$$
Let us take $n_1 = \mu \geq \bar{\mu}$ and a point $p_0\in \, ]0,a[$
which defines a corresponding point $p_1$ as described above.
The proof of our claim starts now.

As a consequence of Lemma \ref{lem-3.2} it will be sufficient to prove
that
$$\frac{b_{n_1}}{p_1} < \cosh\Bigl( \frac{\beta - \alpha}{2} \, \sqrt{g - n_0 \Lambda }\,
\Bigr ):= \kappa,$$
that is, we want to prove the inequality
\begin{equation}\label{eq-3k}
p_1 > \frac{b_{n_1}}{\kappa}\,.
\end{equation}
If $\tfrac{b_{n_1}}{\kappa} \leq a_{n_1}$ we are done. Hence,
it is not restrictive if, from now on, we suppose the opposite inequality.
As a consequence we have that \eqref{eq-3k} is satisfied if and only if the level line
passing through $( \tfrac{b_{n_1}}{\kappa},0)$ has energy less than ${\mathcal E}_1(p_0,0).$

To this aim, we first prove that there is a lower bound $\delta_0$ (depending on $n_1$) such that
\begin{equation}\label{eq-3delta}
\frac{g}{2}\,\bigl(\frac{b_{n_1}}{\kappa}\bigr)^2 - n_1
{\mathcal F}\bigl(\frac{b_{n_1}}{\kappa}\bigr)
\geq \delta_0 > 0.
\end{equation}
Using the fact that ${\mathcal E}_1(b_{n_1},0) = 0$ (by definition of $b_{n_1}$),
the last assertion is equivalent to
$$n_1\,\left(
{\mathcal F}\bigl({b_{n_1}}\bigr) -
{\mathcal F}\bigl(\frac{b_{n_1}}{\kappa}\bigr)\,\right )\,
- \, \frac{g}{2} \, b_{n_1}^2 \,(1 - \tfrac{1}{{\kappa}^2})
\geq \delta_0\,.$$
We distinguish two cases, according to the fact that
$\tfrac{b_{n_1}}{\kappa} > b$ or $\tfrac{b_{n_1}}{\kappa} \leq b.$

In the former case, we can write
\begin{eqnarray*}
n_1\,\left(
{\mathcal F}\bigl({b_{n_1}}\bigr) -
{\mathcal F}\bigl(\frac{b_{n_1}}{\kappa}\bigr)\,\right )\,
- \, \frac{g}{2} \, b_{n_1}^2 \,(1 - \tfrac{1}{{\kappa}^2})
&\geq&
n_1\,\eta \, b_{n_1} \frac{\kappa - 1}{\kappa}
-  \, \frac{g}{2} \, b_{n_1}^2 \,(1 - \tfrac{1}{{\kappa}^2})\\
&\geq&
n_1\,\eta \, b \frac{\kappa - 1}{\kappa}
-  \, \frac{g}{2} \,(1 - \tfrac{1}{{\kappa}^2}):= \delta_0\,,
\end{eqnarray*}
with $\delta_0 > 0$ provided that
$$n_1 > {\mu}^*:= \frac{g(\kappa + 1)}{2 \eta b}\,.$$
In the latter case, we have

\begin{eqnarray*}
n_1\,\left(
{\mathcal F}\bigl({b_{n_1}}\bigr) -
{\mathcal F}\bigl(\frac{b_{n_1}}{\kappa}\bigr)\,\right )\,
- \, \frac{g}{2} \, b_{n_1}^2 \,(1 - \tfrac{1}{{\kappa}^2})
&\geq&
n_1\,
{\mathcal F}\bigl({b_{n_1}}\bigr)
- \, \frac{g}{2} \, b_{n_1}^2 \,(1 - \tfrac{1}{{\kappa}^2})\\
&=&
n_1\,\left(
{\mathcal F}\bigl({b_{n_1}}\bigr) -
{\mathcal F}(b)\,\right )\,
- \, \frac{g}{2} \, b_{n_1}^2 \,(1 - \tfrac{1}{{\kappa}^2})\\
&\geq& n_1\,\eta \, b \frac{\kappa - 1}{\kappa} -  \, \frac{g}{2}
\,(1 - \tfrac{1}{{\kappa}^2})= \delta_0\,.
\end{eqnarray*}

Thus, it remains to find $p_0$ such that ${\mathcal E}_1(p_0,0) \in \,]-\delta_0,0[\,.$
To check this fact, it is equivalent to prove that
$$-{\mathcal E}_1(p_0,0) = - n_1 {\mathcal F}(p_0) +\frac{g}{2}\,p_0^2
< n_1\,\eta \, b \frac{\kappa - 1}{\kappa}
-  \, \frac{g}{2} \,(1 - \tfrac{1}{{\kappa}^2}),$$
that is,
\begin{equation}\label{eq-3p0}
n_1 \Bigl(\, \eta \, b \frac{\kappa - 1}{\kappa} + {\mathcal F}(p_0) \,\Bigr)
>
\frac{g}{2}\Bigl(\, p_0^2 + \tfrac{{\kappa}^2 - 1}{{\kappa}^2}\,\Bigr).
\end{equation}
We define $\hat{p}_0\,\in\,]0,a[$ as the solution of equation
$${\mathcal F}(s) = - \eta \, b \frac{\kappa - 1}{2\kappa}$$
and take $p_0$ such that
$$0 < p_0 \leq \hat{p}_0\,.$$
For $n_1$ satisfying
$$n_1 > \tilde{\mu}:= \frac{g({\kappa}^2(a^2 +1) -1)}{{\kappa}({\kappa} -1)b \eta}$$
we obtain
$$n_1 \Bigl(\, \eta \, b \frac{\kappa - 1}{\kappa} + {\mathcal F}(p_0) \,\Bigr)
\geq
n_1 \,\eta \, b \frac{\kappa - 1}{2\kappa}
>
\frac{g}{2}\Bigl(\, a^2 + \tfrac{{\kappa}^2 - 1}{{\kappa}^2}\,\Bigr)
$$
and \eqref{eq-3p0} is proved.
In this manner, we have found $\hat{p}_0$ and our claim is proved for
$$n_1 > m^*_2:= \max\{{\mu}^*, \tilde{\mu}\}.$$

We are now in position to construct a suitable domain including a topological horseshoe.
This is discussed in the next section.

\section{Proof of the main result}\label{sec-4}
Assume $(H_0)$ and let $n(t)$ be like in \eqref{eq-3nt}, with
$$0 < n_0 < m^*_0 < m^*_2 < n_1\,.$$
Let us fix also a number
$${\bar{p}}_0 \in\, ]0,p^*[\,,\quad \mbox{with } \; p^*:= \min \{{\check{p}_0}, {\hat{p}_0}\}$$
and consider the compact annular region
$${\mathcal M}^1_c:= \{(x,y): 0\leq x \leq 1, \;  c\leq {\mathcal E}_1(x,y) \leq 0\},\;\;
\mbox{ for } c= {\mathcal E}_1({\bar{p}}_0,0).$$
As in Section \ref{sec-3} (see, in particular, Lemma \ref{lem-3.5})
we consider also the set
$${\mathcal N}_c:
={\mathcal M}^1_c\cap (\, [a_{n_1},+\infty)\times[0,+\infty)\,).$$
Now we introduce a polar coordinate system with center at $P_1 = (a_{n_1},0)$
and consider the following compact subsets ${\mathcal N}'_c$ and
${\mathcal N}''_c$ defined by
$${\mathcal N}'_c:= \{z\in {\mathcal N}_c\,:\, \theta(\alpha,z)\in [-\frac{5\pi}{2},-2\pi]\,\}$$
and
$${\mathcal N}''_c:= \{z\in {\mathcal N}_c\,:\, \theta(\alpha,z)\in [-\frac{9\pi}{2},-4\pi]\,\},$$
respectively.
\\
Notice that a point $z$ belongs to ${\mathcal N}'_c$ if and only if the solution $(x(t),y(t))$
of $(E_1)$ with $(x(0),y(0)) = z\in {\mathcal N}_c$ is such that
\begin{equation}\label{eq-4.1}
(x(\alpha),y(\alpha))\in
{\mathcal M}^1_c\cap (\, [a_{n_1},+\infty)\times(-\infty,0]\,)
\end{equation}
and $x(\cdot)$ has exactly two strict maximum points separated by one
strict minimum point along the time interval
$[0,\alpha],$ that is the trajectory crosses the $x-$axis exactly three times.
On the other hand,
$z$ belongs to ${\mathcal N}''_c$ if and only if the solution $(x(t),y(t))$
of $(E_1)$ with $(x(0),y(0)) = z\in {\mathcal N}_c$ satisfies \eqref{eq-4.1}
and $x(\cdot)$ has exactly three strict maximum points separated by two
strict minimum points along the time interval
$[0,\alpha],$ that is the trajectory crosses the $x-$axis exactly five times.

\begin{remark}\label{rem-4.1}
For sake of simplicity in the exposition we have confined ourselves to the
construction of only two sets ${\mathcal N}'_c$ and ${\mathcal N}''_c$ in order to
obtain at the end a dynamics on two symbols. We point out, however, that Lemma \ref{lem-3.5}
permits to define in the same manner (if one takes $n_1$ large enough) an arbitrary number $\ell$
of pairwise disjoint sets and, correspondingly, to obtain at the end a semiconjugation with
a Bernoulli shift on $\ell$ symbols.
\end{remark}

\medskip

Let $\gamma: [0,1]\to {\mathbb R}^2$ be a continuous mapping such that
${\bar{\gamma}}$ crosses ${\mathcal N}_c$ from its inner boundary $L^1_c$ to its outer boundary
$L^1_0\,$.
Without loss of generality (if necessary, we can restrict $\gamma$ to a closed
subinterval of its domain), we assume that
\begin{equation}\label{eq-4.1g}
{\mathcal E}_1(\gamma(0)) = c < 0 = {\mathcal E}_1(\gamma(1)),\quad
\gamma(t) \in {\mathcal N}_c\,,\; \forall\, t\in [0,1].
\end{equation}

\bigskip

\begin{figure}[h]
\quad\psfig{file=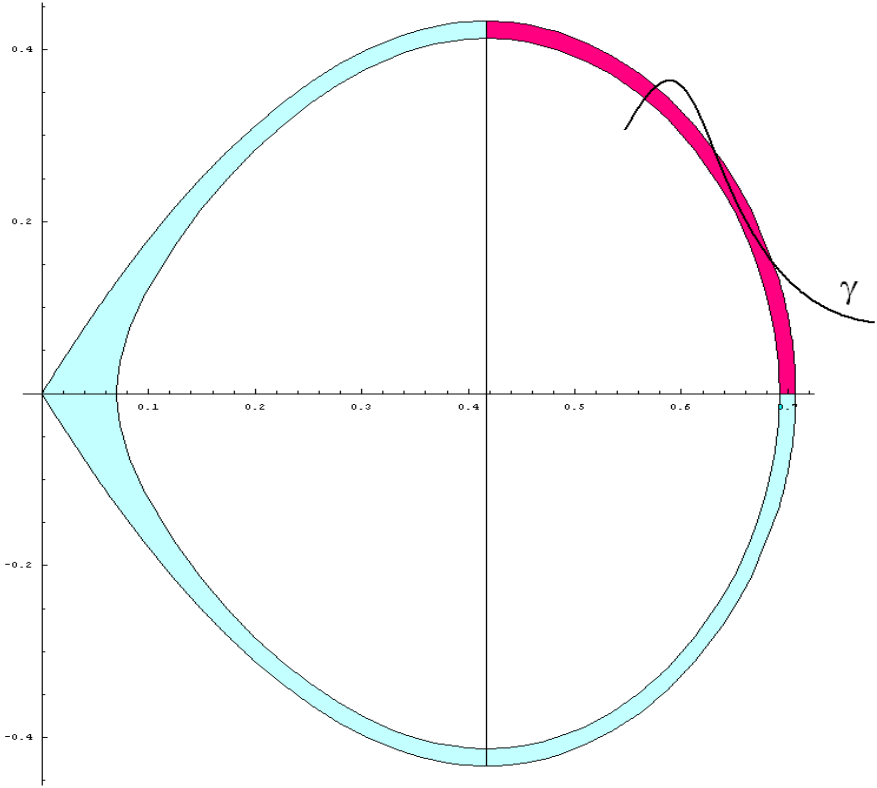,width=5.4in,height=4.2in,angle=0}
\vspace*{0in} \caption{\footnotesize{\em Example of a region ${\mathcal M}^{1}_c$
in which we have put in evidence the subset ${\mathcal N}_c\,.$
The chosen parameters are the same like in Figure 2. We have
taken as the inner boundary of ${\mathcal M}^{1}_c$ the level line passing through
the point $(0.07,0)$ yielding to $c\simeq-0.00850436.$ In this case
we have $a_{n_1} \simeq 0.417157 > a = 0.4$ and ${\mathcal E}_1(a_{n_1},0)\simeq -0.0936779.$
A path $\gamma$ crossing the inner and the outer boundaries of the
set ${\mathcal N}_c$ is also shown. For our argument we consider only the restriction
of $\gamma$ to a subinterval of its domain in order to have condition \eqref{eq-4.1g}
satisfied.}}
\vspace{0.4cm}
\end{figure}

\bigskip

Lemma \ref{lem-3.5} guarantees the existence of $n^* = n^*_{2,c}$
such that for $n_1 > n^*,$
$$\theta(\alpha,z) < - \frac{9 \pi}{2},\;\; \forall\, z\in L_c\cap {\mathcal N}_c\,,$$
follows. On the other hand,
$$\theta(\alpha,z) > -\pi, \;\; \forall\, z\in L_0\cap {\mathcal N}_c\,.$$
If we consider now the composite map
$$[0,1]\ni t\mapsto \theta(\alpha, \gamma(t))\in {\mathbb R},$$
we obtain
$$\theta(\alpha, \gamma(0)) < - \frac{9 \pi}{2} < - \pi < \theta(\alpha, \gamma(1)).$$
A standard continuity argument allows to determine two subintervals
$[t''_1,t''_2]$ and $[t'_1,t'_2]$ of $[0,1]$ with
$0 < t''_1 < t''_2 < t'_1 < t'_2 < 1,$
such that
$$\gamma(t) \in {\mathcal N}'_c\,,\;\forall\, t\in[t'_1,t'_2],
\qquad
\gamma(t) \in {\mathcal N}''_c\,,\;\forall\, t\in[t''_1,t''_2]$$
and, moreover,
\begin{equation}\label{eq-4.2}
\begin{array}{ll}
{\displaystyle{
\theta(\alpha,\gamma(t'_1)) = - \frac{5\pi}{2}\,,\;
\theta(\alpha,\gamma(t'_2)) = - 2\pi,\;}}\\
\\
{\displaystyle{
\theta(\alpha,\gamma(t''_1)) = - \frac{9\pi}{2}\,,\;
\theta(\alpha,\gamma(t''_2)) = - 4\pi.}}
\end{array}
\end{equation}
Let now $p_0$ be an arbitrary, but fixed number such that
$${\bar{p}}_0 < p_0 < p^*.$$
We denote by ${\mathcal W}$ the part of the plane between the
level lines for energy ${\mathcal E}_0$
of system $(E_0)$ passing, respectively, through $(p_0,0)$ and $(p_1,0),$ where
we recall that $(p_1,0)$ is other intersection of the level line
$${\mathcal E}_1(x,y) = {\mathcal E}_1(p_0,0),\quad 0\leq x \leq 1,$$
with the $x-$axis (see Figure 4 below).

\bigskip

\begin{figure}[h]
\quad\psfig{file=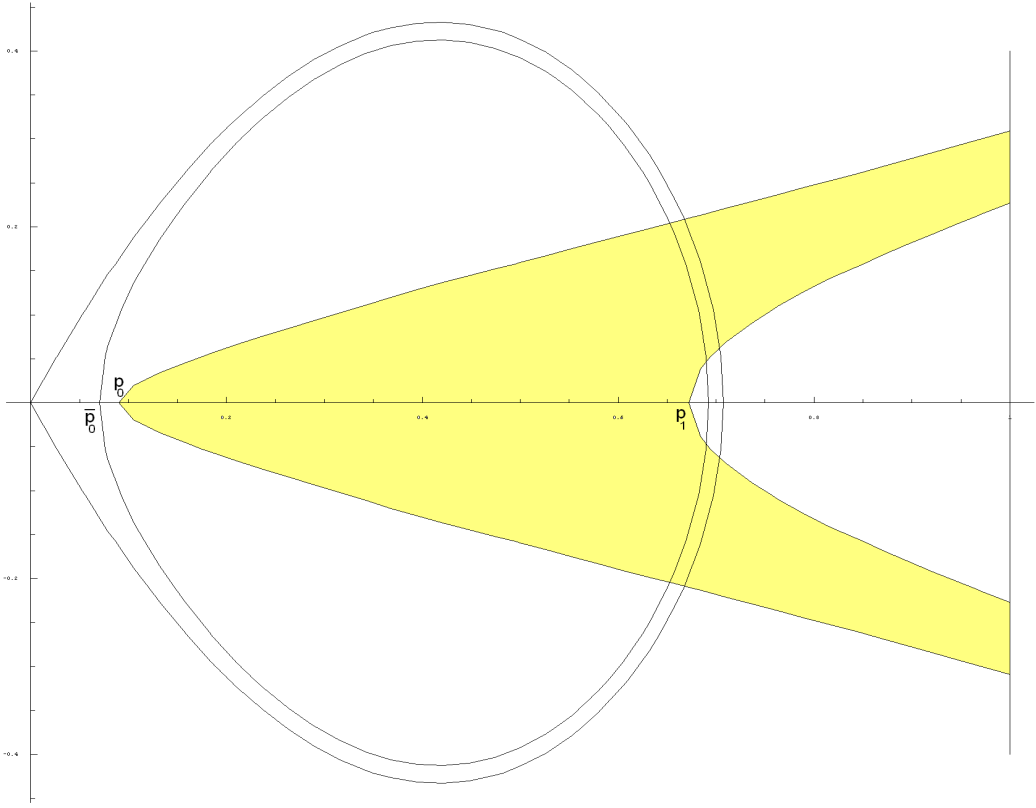,width=5.4in,height=4.2in,angle=0}
\vspace*{0in} \caption{\footnotesize{\em Example of a region
${\mathcal W}\subseteq [0,1]\times {\mathbb R}$ between the
two level lines of energy ${\mathcal E}_0$ passing, respectively, through
$(p_0,0)$ and $(p_1,0).$}}
\vspace{0.4cm}
\end{figure}

\bigskip

The intersection of ${\mathcal M}^1_c$ with ${\mathcal W}$ defines two connected components
$${\mathcal M}^1_c \cap {\mathcal W} = {\mathcal A} \cup {\mathcal B},$$
where we have denoted by ${\mathcal A}$ the component contained in ${\mathcal N}_c$
and by ${\mathcal B}$ its symmetric part with respect to the $x-$axis.
The sets ${\mathcal A}$ and ${\mathcal B}$ are generalized rectangles, that is
they are homeomorphic to the unit square $[0,1]^2.$
\\
In principle,  the set ${\mathcal A}$ (which is the
upper component of ${\mathcal M}^1_c \cap {\mathcal W}$ ) is not necessarily contained in
${\mathcal N}_c\,.$ However, such inclusion occurs when
$n_1$ is sufficiently large, as we have required throughout all the paper.
To check this claim, let us consider the intersection of an arbitrary level line
$${\mathcal E}_0(x,y) = {\mathcal E}_0(s,0), \quad\mbox{for} \, s\in [p_0,p_1]$$
with a level line
$${\mathcal E}_1(x,y) = {\mathcal E}_1(\xi,0), \quad\mbox{for } \,\xi\in [0,{\bar{p}}_0]$$
within the strip $[0,1]\times {\mathbb R}.$ By a symmetry argument, it is
sufficient to look at the intersection points $(x,y)$ with $y > 0.$ This, in turns, yields
to the comparison between the two functions
$$\zeta_0(x,s):=- \frac 1 2 g s^2 + n_0 {\mathcal F}(s) + \frac 1 2 g x^2 - n_0 {\mathcal F}(x)$$
and
$$\zeta_1(x,\xi):=- \frac 1 2 g {\xi}^2 + n_1 {\mathcal F}(\xi) +
\frac 1 2 g x^2 - n_1 {\mathcal F}(x).$$
More precisely, our aim is to prove that
\begin{equation}\label{eq-4.3}
\zeta_1(x,\xi) - \zeta_0(x,s) > 0,\quad\forall\, x\in [{{p}}_0,a_{n_1}],\;
\xi\in [0,{\bar{p}}_0],\; s\in [p_0,p_1].
\end{equation}
Indeed, to check \eqref{eq-4.3} it is sufficient to verify that
\begin{equation}\label{eq-4.3b}
\zeta(x):=\zeta_1(x,{\bar{p}}_0) - \zeta_0(x,p_0) > 0,\quad\forall\, x\in [{{p}}_0,a_{n_1}].
\end{equation}
By the above positions and taking into account that ${\mathcal F}$
is strictly decreasing on $[0,a],$ we have
\begin{eqnarray*}
\zeta(x)&=&
\frac 1 2 g ( {p_0}^2 - {{\bar{p}}_0}^2 )
+ n_0 ({\mathcal F}({\bar{p}}_0) - {\mathcal F}(p_0) )  + (n_0 - n_1)
\bigl( {\mathcal F}(x) - {\mathcal F}({\bar{p}}_0) \bigr)\\
&\geq&
\frac 1 2 g ( {p_0}^2 - {{\bar{p}}_0}^2 )
 + (n_0 - n_1) \int_{{\bar{p}}_0}^{x} F(u)\,du
\end{eqnarray*}
and therefore, we can conclude that
$$\zeta(x)\geq \frac 1 2 g ( {p_0}^2 - {{\bar{p}}_0}^2 ) > 0,\quad\forall\, x\in
[{{{p}}_0},{{\bar{p}}_0}{\!\!\!~}^+],$$
where ${{\bar{p}}_0}{\!\!\!~}^+$ is defined as the unique point satisfying
$$\int_{{{\bar{p}}_0}}^{{{\bar{p}}_0}{\!\!\!~}^+} F(u)\,du = 0,\quad
{{\bar{p}}_0}{\!\!\!~}^+\in\, ]a,1[\,.$$
Hence \eqref{eq-4.3b} is satisfied provided that we have
$$a_{n_1} \leq {{\bar{p}}_0}{\!\!\!~}^+\,.$$
This latter property is true for $n_1$ sufficiently large. In
fact, once that ${{\bar{p}}_0}$ is fixed, we have that
${{\bar{p}}_0}{\!\!\!~}^+ > a$ is fixed too and we already observed that
$a_{\mu} \to a$ as $\mu\to +\infty.$ Thus, having proved our
claim, we know that the generalized rectangles ${\mathcal A}$ and
${\mathcal B}$ look like in Fig. 5.

\begin{figure}[h]
\quad\psfig{file=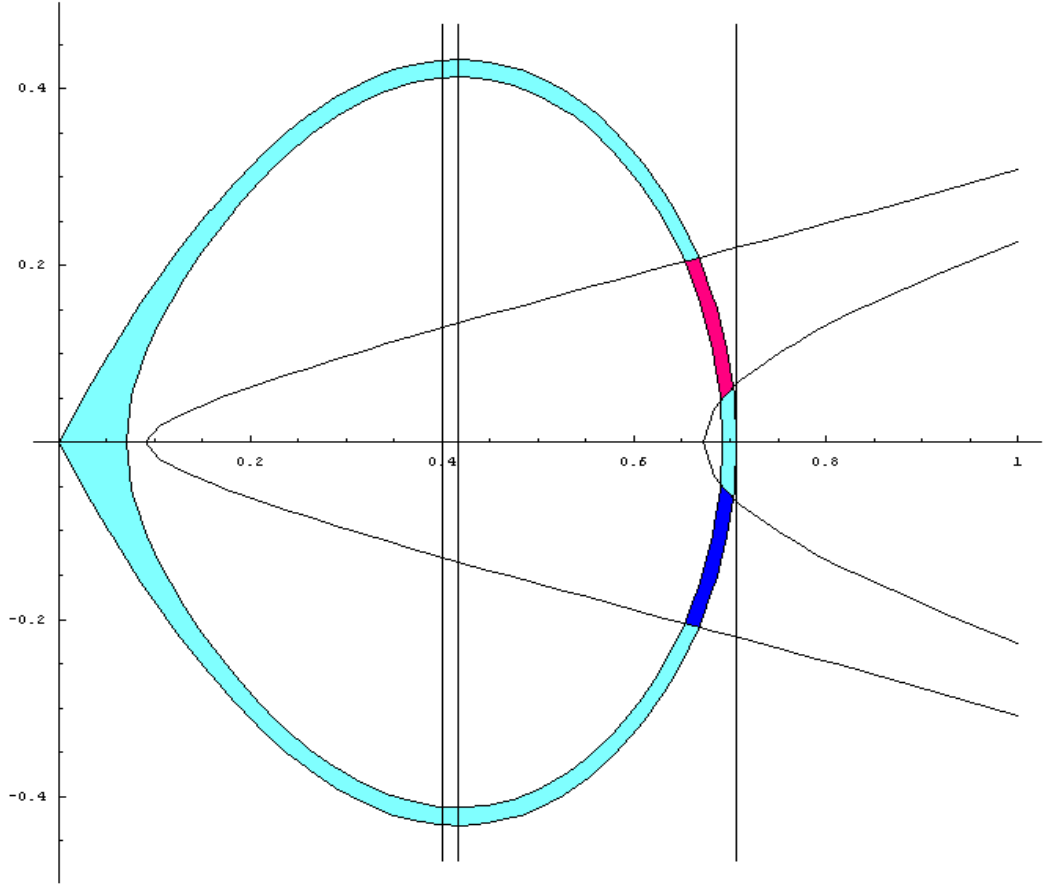,width=7.2in,height=5.6in,angle=0}
\vspace*{0in} \caption{\footnotesize{\em Example of a region ${\mathcal M}^1_c$
in which we have put in evidence the intersection with ${\mathcal W}$
consisting of two generalized rectangles ${\mathcal A}$
(the upper component) and ${\mathcal B}$ (the lower component).
The parameters are the same like in Figure 2 and Figure 3.
Moreover, we have also chosen ${{\bar{p}}_0} = 0.07$ and obtained
${{\bar{p}}_0}{\!\!\!~}^+ \simeq 0.652494.$
In the picture we have drawn the vertical lines
$x=a=0.4$ and $x=a_{n_1}\simeq 0.417157$
to show that
${\mathcal A}\subseteq {\mathcal N}_c\,.$
}}
\vspace{0.4cm}
\end{figure}

\bigskip

We fix now an orientation for sets ${\mathcal A}$ and ${\mathcal
B}.$ More precisely, given a generalized rectangle ${\mathcal R}$
(also called a two-dimensional cell) we select on its boundary two
disjoint compact arcs ${\mathcal R}^-_l$ and ${\mathcal R}^-_r\,,$
like in Section \ref{sec-2}. Such arcs will be conventionally
called the {\itshape{left }} and the {\itshape{right}} sides of
${\mathcal R}.$ We also set ${\mathcal R}^-:= {\mathcal R}^-_l\cup
{\mathcal R}^-_r\,.$
\\
In the case of the sets ${\mathcal A}$ and ${\mathcal B},$ we put
$${\mathcal A}^-:= {\mathcal A}\cap \partial {\mathcal M}^1_c\,,\quad
{\mathcal B}^-:= {\mathcal B}\cap \partial {\mathcal W}.$$
Accordingly, we denote by ${\mathcal A}^-_l$ and ${\mathcal A}^-_r$
the two components of ${\mathcal A}^-$ and, similarly, by
${\mathcal B}^-_l$ and ${\mathcal B}^-_r$
the two components of ${\mathcal B}^-$ (the order is not relevant).
At last, we also define the disjoint compact sets
\begin{equation}\label{eq-dd}
{\mathcal D}_1:= {\mathcal A} \cap{\mathcal N}'_c\,,
\quad {\mathcal D}_2:= {\mathcal A}\cap {\mathcal N}''_c\,.
\end{equation}
Let
$\psi: {\mathbb R}^2 \to {\mathbb R}^2$
be the Poincar\'{e} operator associated to system \eqref{eq-3.2}, that is
$z\mapsto \zeta(\beta,z),$ where $\zeta(t) = \zeta(t,z)$ is the solution
of system \eqref{eq-3.2} with $\zeta(0) = z.$ Observe that
$$\psi = \psi_0 \circ \psi_1\,,$$
where
$\psi_1$ is the Poincar\'{e} map associated to system $(E_1)$ for the
time interval $[0,\alpha]$ and
$\psi_0$ is the Poincar\'{e} map associated to system $(E_0)$ for the
time interval $[0,\beta - \alpha].$

As a consequence of \eqref{eq-4.2} and the fact that the annulus ${\mathcal M}^1_c$
is invariant with respect to the dynamical system associated to $(E_1),$ we have
$$({\mathcal D}_i, \psi_1) : ({\mathcal A}, {\mathcal A}^-) \stretch
({\mathcal B}, {\mathcal B}^-),\quad\mbox{for } i= 1,2.$$
On the other hand, we also have that
$$({\mathcal B}, \psi_0) : ({\mathcal B}, {\mathcal B}^-) \stretch
({\mathcal A}, {\mathcal A}^-).$$
Thus, we can conclude that
$$({\mathcal D}_i, \psi) : ({\mathcal A}, {\mathcal A}^-) \stretch
({\mathcal A}, {\mathcal A}^-),\quad\mbox{for } i= 1,2.$$

The results in \cite{KeYo-01} and \cite{PaZa-04a}, \cite{PaZa-04b} guarantee the
existence of a compact invariant set
$$\Lambda \subseteq {\mathcal D}_1 \cup {\mathcal D}_2 \subseteq {\mathcal A}$$
for which $\psi|_{\Lambda}$ is semiconjugate to a two-sided Bernoulli shift.
Moreover, for each periodic sequence of two symbols there is a point in $\Lambda$
which is periodic. In fact, all the assumptions of Lemma \ref{lem-2.1}
are satisfied with respect to the Poincar\'{e} map $\psi,$ the oriented rectangle
$({\mathcal R},{\mathcal R}^-):= ({\mathcal A},{\mathcal A}^-)$ and the sets
${\mathcal D}_i$ ($i=1,2$) defined in \eqref{eq-dd}.

As a final step, we want to give an interpretation of the above statement
in terms of the solutions to equation
$$x'' - g x + n(t)F(x) = 0.$$
Consider a two-sided sequence of two symbols $(i_k)_{k\in {\mathbb Z}}$
with
$$i_k \in \{1,2\},\quad\forall\, k\in {\mathbb Z}.$$
By the semiconjugation, we know that there is at least one point
$(x_0,y_0)= z_0\in \Lambda$ such that
$$\psi^{k}(z_0)\in {\mathcal D}_{i_k}\,,\quad\forall\, k\in {\mathbb Z}.$$
Let $x(\cdot)$ (globally defined on ${\mathbb R}$)
be the solution of the differential equation with
$x(0) = x_0$ and $x'(0) = y_0\,.$ For such a solution, we have that
\begin{equation}\label{eq-4.k}
(x(k\beta),x'(k\beta))\in {\mathcal D}_{i_k}\,,\quad\forall\, k\in {\mathbb Z}.
\end{equation}
By construction,
$$0 < x <1,\quad\forall\, (x,y)\in {\mathcal D}_{1}\cup {\mathcal D}_{2}\,.$$
Hence, Lemma \ref{lem-3.0} ensures that
$$0 < x(t) < 1,\quad\forall\, t\in {\mathbb R}$$
and thus we have that $x(\cdot)$ is a solution of the original equation
\eqref{eq-3.1} no matter which kind of modification has been performed
on $F$ outside the interval $[0,1].$

Now, we explain the meaning of \eqref{eq-4.k} by analyzing the following two possibilities:
\begin{itemize}
\item{}
$(x(k\beta),x'(k\beta))\in {\mathcal D}_{1}\,,$ for some $k\in {\mathbb Z};$
\item{}
$(x(k\beta),x'(k\beta))\in {\mathcal D}_{2}\,,$ for some $k\in {\mathbb Z}.$
\end{itemize}
In the former case, $x(t)$ has precisely two strict maximum points separated by
one strict minimum point along the time interval
$[(k-1)\beta,(k-1)\beta +\alpha],$
while, in the latter case, $x(t)$ has precisely three strict maximum points separated by
two strict minimum points along the time interval
$[(k-1)\beta,(k-1)\beta +\alpha].$
In both the situations, $x(\cdot)$ is convex in the interval $[(k-1)\beta +\alpha,k\beta],$
with $x'((k-1)\beta + \alpha) < 0$ and $x'(k\beta) > 0.$

At last, we observe that if the sequence $(i_k)_{k\in {\mathbb Z}}$ is
periodic, that is $i_k = i_{k+\ell}$ for some $\ell \geq 1,$ then we can take
$x(\cdot)$ as a $\ell \beta$-periodic solution as well.

The proof of our main theorem is complete.
\hfill$\qed$

\bigskip

The proof of Theorem \ref{th-1.2} relies on slight modifications of the arguments described above
and therefore it is omitted.

Both Theorem \ref{th-1.1} and Theorem \ref{th-1.2} are stable with respect
to small (in the $L^1$-norm) perturbations of the weight function $n(\cdot).$

\section{Appendix}\label{sec-5}
In this section we give a proof of Lemma \ref{lem-3.0}. Actually, we present a more
general result which improves \cite[Lemma 2.1]{ZaZa-05}
and may have some independent interest.

We consider the second order scalar ODE
\begin{equation}\label{eq-5.1}
x'' + h(t,x) = 0
\end{equation}
where $h: [t_0,t_1]\times{\mathbb R}\to {\mathbb R}$ is a Carath\'{e}odory function,
that is, we assume that
$h(\cdot,s)$ is measurable for all $s\in {\mathbb R},$
$h(t,\cdot)$ is continuous for almost every $t\in [t_0,t_1]$ and,
for every $r > 0$ there is a measurable function $\rho_r\in L^1([t_0,t_1],{\mathbb R}^+)$
such that $|h(t,s)|\leq \rho_r(t)$ for almost every $t\in [t_0,t_1]$ and
every $s\in [-r,r].$ Solutions of $(\ref{eq-5.1})$
are considered in the Carath\'{e}odory
sense too (cf. \cite[p.28]{Ha-80}).

\begin{lemma}\label{lem-5.1} Suppose
$$
h(t,0) \equiv 0
$$
and, for a.e. $t\in [t_0,t_1],$
$$
h(t,s) > 0, \;\forall \, s < 0,\qquad
h(t,s) < 0, \;\forall \, s \geq 1.
$$
Let $x(\cdot)$ be a solution of \eqref{eq-5.1} defined on $[t_0,t_1]$ and
such that
$$0 < x(t_0), x(t_1) < 1.$$
Then
$0\leq x(t) \leq 1,\,\forall\, t\in [t_0,t_1].$ Moreover,
$x(t)>0,\,$ $\forall\, t\in [t_0,t_1]$ if
$h$ is (locally) lipschitzian at $s=0$ and
$x(t)< 1,\,$ $\forall\, t\in [t_0,t_1],$
if $h$ is (locally) lipschitzian at $s=1.$
\end{lemma}
\begin{proof}
At first we prove that $x(t) \geq  0$ for all $t\in [t_0,t_1].$
If, by contradiction $x({\tilde{t}}) < 0$ for some ${\tilde{t}}\in [t_0,t_1],$
then we can find $\sigma_0\,,\sigma_1$ with
$t_0 < \sigma_0 < {\tilde{t}} < \sigma_1 < t_1$
such that
$x(\sigma_0) = x(\sigma_1) = 0$ and $x(t) < 0$
for all $t\in \,]\sigma_0,\sigma_1[\,.$
Multiplying equation
\begin{equation}\label{eq-5.2}
- x''(t) = h(t,x(t))
\end{equation}
by $x(t)$  and then integrating on $[\sigma_0,\sigma_1],$ we
obtain $ \int_{\sigma_0}^{\sigma_1} x'(t)^2\,dt = \break
\int_{\sigma_0}^{\sigma_1} h(t,x(t))x(t)\,dt < 0, $ a
contradiction. A similar computation shows that $v(x)\leq 1$ for
all $x\in [t_0,t_1].$ In fact, if by contradiction $x({\hat{t}}) >
1$ for some ${\hat{t}}\in [t_0,t_1],$ then we can find
$\tau_0\,,\tau_1$ with $t_0 < \tau_0 < {\hat{t}} < \tau_1 < t_1$
such that $x(\tau_0) = x(\tau_1) = 1$ and $x(t) > 1$ for all $t\in
\,]\tau_0,\tau_1[\,.$ Multiplying equation \eqref{eq-5.2} by
$x(t)-1$  and then integrating on $[\tau_0,\tau_1],$ we obtain $
\int_{\tau_0}^{\tau_1} x'(t)^2\,dt = \break \int_{\tau_0}^{\tau_1}
h(t,x(t))( x(t) - 1)\,dt < 0, $ a contradiction.
\\
If we have $\min x(t) = x(t_*) = 0,$ then also $x'(t_{*}) = 0$
since $t_{*}\in \,]t_0,t_1[\,.$
Hence $x(\cdot)$ is a solution of the initial value problem
$$
x'' + h(t,x) = 0,\,\quad x(t_{*}) = x'(t_{*}) = 0
$$
and therefore, if $h$ satisfies a Lipschitz condition at $s=0,$
we must have $x(t) = 0,\forall\, t\in [t_0,t_1]$ (a contradiction with
our hypotheses).
\\
Assume now that $h$ satisfies a Lipschitz condition at $s =1$ and
$\max x(t) = x(t^*) = 1.$ In this case, $x'(t^*) = 0$
because $t^{*}\in \,]t_0,t_1[\,.$
Define also
the auxiliary function
$\displaystyle{
{\tilde h}(t,s) := h(t,\max\{1,s\})
}$
which is lipschitzian at $s =1$ too
and let $y(t)$ be a solution of the
Cauchy problem
$$
x'' + {\tilde h}(t,x) = 0,\,\quad x(t^*) = 1,\;  x'(t^*) = 0
$$
defined on a maximal interval of existence. Using the fact that
${\tilde h}(t,s) < 0$ for all $s\in{\mathbb R}$ and almost every $t\in [t_0,t_1],$
we find that $y(t) > 1$ for all $t \not= t^*$ in the domain
of $y(\cdot)$ and therefore
$y(\cdot)$ is a solution of \eqref{eq-5.1} as well. By the local uniqueness
of the solutions to the Cauchy problem under consideration we conclude that
$x(\cdot) = y(\cdot)$ in a neighborhood of $t^*$ (a contradiction to
$x(t) \leq 1 < y(t)$ for $t\not=t^*$).
\end{proof}

\noindent
We remark that the same proof works for the slightly more general equation
$$x'' + c x'+ h(t,x) = 0,$$
under the same assumptions on $h$ and for every $c\in {\mathbb R}.$

\end{document}